\documentclass[a4paper,12pt]{article}

\usepackage[left=2cm,right=2cm, top=2cm,bottom=3cm,bindingoffset=0cm]{geometry}

\usepackage{verbatim}
\usepackage{amsmath}
\usepackage{amsthm}
\usepackage{amssymb}
\usepackage{delarray}
\usepackage{cite}
\usepackage{hyperref}
\usepackage{mathrsfs}
\usepackage{tikz}
\usetikzlibrary{patterns}
\usepackage{caption}
\DeclareCaptionLabelSeparator{dot}{. }
\newcommand{\al}{\alpha}

\newcommand{\de}{\delta}
\newcommand{\la}{\lambda}

\newcommand{\eps}{\varepsilon}

\theoremstyle{plain}

\numberwithin{equation}{section}

\newtheorem{thm}{Theorem}[section]
\newtheorem{lem}[thm]{Lemma}
\newtheorem{prop}[thm]{Proposition}
\newtheorem{cor}[thm]{Corollary}

\theoremstyle{definition}

\newtheorem{df}[thm]{Definition}

\theoremstyle{remark}

\newtheorem{remark}[thm]{Remark}

\DeclareMathOperator*{\Res}{Res}

 \allowdisplaybreaks

\begin{document}

\begin{center}
{\Large\bf Regularization and inverse spectral problems \\[0.2cm] for differential operators with distribution coefficients}
\\[0.5cm]
{\bf Natalia P. Bondarenko}
\end{center}

\vspace{0.5cm}

{\bf Abstract.} In this paper, we consider a class of matrix functions, which contains regularization matrices of Mirzoev and Shkalikov for differential operators with distribution coefficients of order $n \ge 2$. We show that every matrix function of this class is associated with some differential expression. Moreover, we construct the family of associated matrices for a fixed differential expression. Furthermore, our regularization results are applied to inverse spectral theory. We study a new type of inverse spectral problems, which consist in the recovery of distribution coefficients from the spectral data independently of the associated matrix. The uniqueness theorems are proved for the inverse problems by the Weyl-Yurko matrix and by the discrete spectral data. As an example, we consider the case $n = 2$ in detail. 

\medskip


\medskip

  
\vspace{1cm}

\section{Introduction} \label{sec:intro}

This paper is concerned with regularization and inverse spectral problems for differential operators generated by the expression
\begin{align} \nonumber
    \ell_n(y) = & y^{(n)} + \sum_{k = 0}^{m-1 + s} (-1)^k \bigl(\tau_{2k}(x) y^{(k)}\bigr)^{(k)} \\ \label{defln} & + \sum_{k = 0}^{m-1} (-1)^{k+1} \Bigl( \bigl( \tau_{2k+1}(x) y^{(k)}\bigr)^{(k+1)} + \bigl( \tau_{2k+1}(x) y^{(k+1)} \bigr)^{(k)}\Bigr), \quad x  \in (0, 1),
\end{align}
where $n = 2m+s$, $m \in \mathbb N$, $s \in \{ 0, 1\}$, $(\tau_{\nu})_{\nu = 0}^{n-1}$ are distributional coefficients (generalized functions),
$\tau_{\nu} \in W_{2-s}^{-i_{\nu}}[0,1]$ for $\nu = \overline{0,n-1}$, and the singularity orders $(i_{\nu})_{\nu = 0}^{n-1}$ are defined as follows:
\begin{equation} \label{defi}
i_{2k+j} := m-k-j, \quad k \ge 0, \: j \in \{0, 1\}.
\end{equation}
In other words,
\begin{equation} \label{tausig}
\tau_{\nu} = (-1)^{i_{\nu}}\sigma_{\nu}^{(i_{\nu})}, \quad
\nu = \overline{0,n-1},
\end{equation}
where $\sigma_{\nu} \in L_{2-s}[0, 1]$.

In recent years, spectral theory and related issues for linear ordinary differential operators with distribution coefficients have been rapidly developed. In 2016, Mirzoev and Shkalikov \cite{MS16} have proposed a regularization approach for even-order differential operators with distribution coefficients. In particular, their approach allows to reduce the equation $\ell_n(y) = \la y$, where $\la$ is the spectral parameter, to the equivalent first-order system
\begin{equation} \label{sysMS}
Y' = (F(x) + J) Y, \quad x \in (0,1),
\end{equation}
where $Y(x)$ is a column vector-function of size $n$, 
\begin{equation} \label{defJ}
J = \sum_{k = 1}^{n-1} E_{k,k+1} + \la E_{n,1},
\end{equation}
$E_{k,j}$ denotes the constant matrix whose entry at position $(k,j)$ equals $1$ and all the other entries equal zero, $F(x) = [f_{k,j}(x)]_{k,j = 1}^n$ is the so-called \textit{associated matrix} for the differential expression $\ell_n(y)$, $f_{k,j} = 0$ for $j > k$ and $f_{k,j} \in L_1[0,1]$ otherwise.

Analogous results were obtained for the odd-order case in \cite{MS19}.
It is worth mentioning the the reduction of differential equations with regular (integrable) coefficients to the first-order systems of form \eqref{sysMS} by introducing quasi-derivatives is well-known (see \cite{Nai68, EM99}). Weidmann \cite{Weid87} applied such reduction to a specific class of higher-order operators generated by matrix differential expressions, which
included \eqref{defln} with $i_{2k} = 1$, $i_{2k+1} = 0$.

Another regularization approach, based on quadratic forms, was developed by Neiman-Zade and Shkalikov \cite{NZS99, NZS06} for the both ordinary and partial differential equations. Relying on the ideas of \cite{NZS99}, Vladimirov found an associated matrix for a fourth-order operator in \cite{Vlad04} and obtained convenient formulas for construction of associated matrices in the general case by using the coefficients of bilinear forms in \cite{Vlad17}. Here, we focus on the bibliography for higher orders $n > 2$. For $n = 2$, four different regularization approaches  are described in \cite{SS03}.

Regularization of differential equations with distribution coefficients opened a perspective of investigating solution properties and spectral theory for such equations. Relying on the reduction to the first-order system~\eqref{sysMS}, Savchuk and Shkalikov \cite{SS20} constructed the Birkhoff-type solutions for differential equations with distribution coefficients. Konechnaya et al \cite{KM19, KMS23} applied the regularization approach to study the asymptotics of solutions for differential equations on the half-line $(0,\infty)$ as $x \to \infty$. Vladimirov et al \cite{Vlad16, VS21} investigated oscillation properties for higher-order boundary value problems with distribution coefficients.  Using the regularization methods of \cite{MS16, MS19, Vlad17}, Bondarenko \cite{Bond21, Bond22, Bond23-mmas, Bond23-results} has obtained a series of results on inverse spectral problems. Such problems consist in recovering coefficients of differential operators from spectral data. 

Inverse spectral theory has a long history. Classical results in this field were obtained for the Sturm-Liouville operators $-y'' + q(x) y$ with integrable potentials $q$ by using the famous transformation operator method (see the monographs \cite{Mar77, Lev84, FY01, Krav20} and references therein). For distributional potentials of classes $W_2^{\al}$, $\al \ge -1$, inverse problems also have been studied fairly completely (see, e.g., \cite{HM-sd, HM-2sp, FIY08, SS10, Hryn11, Eckh15, Bond21-tamkang, Bond21-mn}). However, inverse problems for higher-order ($n > 2$) differential operators are essentially different, because the transformation operator method is ineffective for them. Therefore, Yurko \cite{Yur92, Yur00, Yur02} has developed \textit{the method of spectral mappings}, which allowed him to create the inverse spectral theory for higher-order differential operators with regular coefficients. In recent years, the ideas of the method of spectral mappings were extended to operators generated by the differential expression \eqref{defln} with $\tau_{n-1} = 0$ (see \cite{Bond21, Bond22, Bond23-mmas, Bond23-results, FIY08, Bond21-tamkang}). In particular, in \cite{FIY08, Bond21-tamkang}, the method of spectral mappings has been transferred to the Sturm-Liouville operators with distribution potentials of class $W_2^{-1}[0,1]$. In \cite{Bond21}, the uniqueness theorems of inverse spectral problems have been proved for the higher-order differential operators with distribution coefficients of the Mirzoev-Shkalikov class \cite{MS16, MS19} on a finite interval. In \cite{Bond23-mmas}, differential operators on the half-line with singular coefficients of various singularity orders have been considered. For those operators, associated matrices have been constructed and the uniqueness theorems have been obtained. In \cite{Bond22}, a constructive approach to the recovery of higher-order differential operators with distribution coefficients from the spectral data has been developed. That approach allowed the author to obtain the necessary and sufficient conditions for solvability of the inverse problem for the third-order differential equation in \cite{Bond23-results}.

% Applications

Spectral theory of linear differential operators has a variety of applications. The second-order Sturm-Liouville (one-dimensional Schr\"odinger) operators are widely used in mechanics, geophysics, acoustics, material science, engineering. In particular, the Sturm-Liouville operators with singular potentials of class $W_2^{-1}[0,1]$ model particle interactions in quantum mechanics \cite{Alb05}. The third order differential operators arise in the study of flows of thin viscous films over solid surfaces \cite{BP96} and in the integration of the nonlinear Boussinesq equation by the inverse scattering transform \cite{McK81}. Inverse spectral problems for the fourth-order differential operators appear in geophysics \cite{Bar74} and in vibration theory \cite{Glad05}. Some six-order eigenvalue problems that occur in mathematical models of vibrations of curved arches were considered in \cite{MZ13}. In recent years, the interest of scholars to spectral properties of the third- and the fourth-order differential operators with non-smooth and distribution coefficients has increased (see, e.g., \cite{UB20, BK21, ZAB22, ZLW23, Pol23}). Thus, the investigation of higher-order differential operators with distribution coefficients, on the one hand, is useful for the development of mathematical methods for a wider range of applied problems. On the other hand, construction of the general spectral theory for such operators is a fundamental mathematical question.

The goal of this paper is to study various associated matrices for the differential expression \eqref{defln}. As a simple example, consider the Sturm-Liouville equation
\begin{equation} \label{StL-intr}
y'' - q(x) y = \la y, \quad x \in (0,1).
\end{equation}
If $q = \sigma' \in W_2^{-1}[0,1]$, then equation \eqref{StL-intr} is equivalent to the system \eqref{sysMS} with the associated matrix $F_1(x) = \begin{bmatrix} \sigma & 0 \\ -\sigma^2 & -\sigma\end{bmatrix}$. On the other hand, if $q \in L_1[0,1]$, then the associated matrix $F_2(x) = \begin{bmatrix} 0 & 0 \\ q & 0 \end{bmatrix}$ can be used. Consequently, the following questions arise:

\begin{enumerate}
\item Are there any other associated matrices and how to describe all the possible associated matrices?
\item Does the choice of the associated matrix influence the spectral characteristics, which are used in the inverse spectral theory, and the results concerning inverse problems?
\end{enumerate}

For the second order, the answers are given in Section~\ref{sec:ex}). However, for higher orders, the situation is much more complicated. Note that the studies of Mirzoev and Shkalikov \cite{MS16, MS19} provide only a specific construction of the associated matrices and do not answer these questions. In the papers \cite{Bond21, Bond22, Bond23-mmas, Bond23-results}, inverse spectral problems are investigated also by using specific forms of associated matrices.

In this paper, we describe the family of all the matrix functions $F(x)$ associated with the differential expression $\ell_n(y)$ of form \eqref{defln} in a certain natural class $\mathfrak F_n$, which is defined in Sections~\ref{sec:even} and~\ref{sec:odd} for even and odd $n$, respectively. Moreover, we prove that every matrix $F(x)$ of $\mathfrak F_n$ is associated with some differential expression $\ell_n(y)$ (see Theorem~\ref{thm:invF}). Furthermore, we show that any two matrices $F(x)$ and $\tilde F(x)$ associated with the same differential expression $\ell_n(y)$ generate equal domains $\mathcal D_F$ and $\mathcal D_{\tilde F}$ for solutions of the equation $\ell_n(y) = \la y$ (see Theorem~\ref{thm:dom}).

Next, we apply the obtained regularization results to inverse spectral problems. We study the influence of the associated matrix choice on the spectral data. As the main spectral characteristics, we use the Weyl-Yurko matrix $M(\la)$, which was introduced by Yurko \cite{Yur92, Yur00, Yur02} and used by Bondarenko \cite{Bond21, Bond22, Bond23-mmas, Bond23-results} for the case of distribution coefficients. The Weyl-Yurko matrix is closely related to several spectra (see \cite{Bond21} for details). In addition, we consider the discrete spectral data which consists of the Weyl-Yurko matrix poles $\Lambda$ and of the so-called weight matrices $\mathcal N(\la_0)$, $\la_0 \in \Lambda$, which are obtained from the Laurent series of $M(\la)$. It is convenient to use the discrete data $\{ \la_0, \mathcal N(\la_0) \}_{\la_0 \in \Lambda}$ for constructive solution of the inverse problem (see \cite{Bond22}). We show that the Weyl-Yurko matrix, in general, depends on the choice of the associated matrix $F(x)$. Namely, if $M(\la)$ and $\tilde M(\la)$ are the Weyl-Yurko matrices that are obtained from different regularizations of the same differential expression $\ell_n(y)$, then $M(\la) = L \tilde M(\la)$, where $L$ is a constant lower-triangular matrix (see Theorem~\ref{thm:L}). The discrete spectral data, on the contrary, do not depend on the associated matrix. Moreover, we prove Theorems~\ref{thm:uniq} and~\ref{thm:sd} on the uniqueness of recovering the coefficients $(\tau_{\nu})_{\nu = 0}^{n-2}$ from the Weyl-Yurko matrix given up to a lower-triangular matrix factor $L$ and from the spectral data $\{ \la_0, \mathcal N(\la_0) \}_{\la_0 \in \Lambda}$, respectively. We emphasize that the considered inverse spectral problems are fundamentally novel comparing with the ones of \cite{Bond21, Bond22, Bond23-mmas}. The results on inverse problems in this paper are independent of the choice of the associated matrix, while in the previous studies, specific associated matrices were considered and the antiderivatives $(\sigma_{\nu})_{\nu = 0}^{n-2}$ were reconstructed. These two types of inverse problems are different and they both generalize the classical inverse problems for differential operators with regular coefficients.

It is worth noting that the regularization methods in \cite{MS16, MS19, Vlad17} were developed for differential expressions of more general forms with locally integrable or locally square integrable antiderivatives $\sigma_{\nu}$ and with non-trivial functional coefficients at $y^{(n)}$. However, in the inverse problem theory, it is natural to consider the form \eqref{defln} with the coefficient $1$ at $y^{(n)}$ and the zero coefficient at $y^{(n-1)}$, following the previous studies \cite{Yur00, Yur02}. The integrability of the antiderivatives $\sigma_{\nu}$ on the whole interval $[0,1]$ is crucial for the asymptotics of the Birkhoff-type solutions, which are important for investigation of the inverse problems. Therefore, in this paper, we confine ourselves to differential expressions of form \eqref{defln} satisfying \eqref{tausig} with $\sigma_{\nu} \in L_{2-s}[0,1]$. However, our regularization results can be transferred to the case of $L_{2-s,loc}(0,1)$ with necessary technical modifications.

The paper is organized as follows. Section~\ref{sec:even} contains the main regularization results together with their proofs for the even-order case. The odd-order case is considered in Section~\ref{sec:odd}. Section~\ref{sec:ip} is concerned with inverse spectral problems. In Section~\ref{sec:ex}, the main results are illustrated by the example of $n = 2$. Section~\ref{sec:concl} contains a brief summary of the results and concluding remarks.

\section{Even order} \label{sec:even}

Let $n = 2m$, $m \in \mathbb N$. Consider the differential expression \eqref{defln} for $s = 0$ with the complex-valued distributional coefficients
$\mathcal T := (\tau_{\nu})_{\nu = 0}^{n-1}$, which belong to the space
$$
\mathfrak T_n := \bigl\{ \mathcal T = (\tau_{\nu})_{\nu = 0}^{n-1} \colon \tau_{\nu} \in W_2^{-i_{\nu}}[0,1], \, \nu = \overline{0,n-1} \bigr\},
$$
where the singularity orders $(i_{\nu})_{\nu = 0}^{n-1}$ are defined by \eqref{defi}.
We will write that $\Sigma = (\sigma_{\nu})_{\nu = 0}^{n-2} = \Sigma(\mathcal T)$ if the relations \eqref{tausig} hold. Note that the antiderivaties $\Sigma$ are not uniquely determined by $\mathcal T$. Anyway, the arguments below are valid for any possible choice of $\Sigma$.

\subsection{Regularization of Mirzoev and Shkalikov} \label{sub:MS}

In this subsection, we provide the construction of the associated matrix of Mirzoev and Shkalikov \cite{MS16} for the even-order differential expression $\ell_n(y)$ by the method of \cite{Vlad17, Bond23-mmas}. 

Denote by $\mathfrak D = C_0^{\infty}(0,1)$ and $\mathfrak D'$ the spaces of test functions and generalized functions, respectively. In other words, $\mathfrak D$ is the space of infinitely differentiable functions $f$ with $\mbox{supp} f \subset (0,1)$ and $\mathfrak D'$ is the space of continuous linear functionals on $\mathfrak D$. For $f \in \mathfrak D'$ and $z \in \mathfrak D$, we use the notation $(f,z) = fz$. In particular, $(f,z) = \int_0^1 f(x) z(x) \, dx$ if $f \in L_{1,loc}(0,1)$.

By direct calculations, one can easily prove the following proposition (see, e.g., \cite{Bond23-mmas}).

\begin{prop} \label{prop:Q}
For $y \in W_2^m[0,1]$, we have $\ell_n(y) \in \mathfrak D'$ and the following relation holds:
\begin{equation} \label{quad}
(\ell_n(y), z) = (-1)^m ( y^{(m)}, z^{(m)} ) + \sum_{r,j = 0}^m (q_{r,j} y^{(r)}, z^{(j)}), \quad z \in \mathfrak D,
\end{equation}
where $\Sigma = \Sigma(\mathcal T)$,
\begin{gather} \label{defQ}
[q_{r,j}]_{r,j = 0}^m = \mathscr Q_n(\Sigma) := \sum_{\nu = 0}^{n-1} \sigma_{\nu}(x) \chi_{\nu, i_{\nu}}, \\ \label{defchi}
\chi_{\nu, i} = [\chi_{\nu,i;r,j}]_{r,j = 0}^m, \quad
\begin{array}{rl}
\chi_{2k, i; s + k, i - s + k} = & C_i^s, \quad s = \overline{0, i}, \\
\chi_{2k+1, i; s+k, i + 1- s + k} = & C_{i + 1}^s - 2 C_i^{s-1}, \quad s = \overline{0, i+1}, 
\end{array}
\end{gather}
all the other entries $\chi_{\nu, i; r,j}$ equal zero, and $C_i^s = \frac{i!}{s!(i-s)!}$ are the binomial coefficients, $C_i^{-1} := 0$.
\end{prop}

Using the entries of the matrix function $Q(x) = [q_{r,j}]_{r,j = 0}^m$ defined by \eqref{defQ}, construct the matrix function $F(x) = [f_{k,j}]_{k,j =1}^n$ by the rule $F = \mathscr S_n(Q)$ given by the formulas
\begin{equation} \label{Seven}
\begin{cases}
    f_{m,j} := (-1)^{m+1} q_{j-1,m}, \quad j = \overline{1, m}, \\
    f_{k,m+1} := (-1)^{k+1} q_{m,2m-k}, \quad k = \overline{m+1, 2m}, \\
    f_{k,j} := (-1)^{k+1} q_{j-1,2m-k} + (-1)^{m+k} q_{j-1,m} q_{m,2m-k}, \quad k = \overline{m+1,2m}, \, j = \overline{1,m}, \\
    f_{k,j} := 0, \quad k < m \:\: \text{or} \:\: j > m+1 \:\: \text{or} \:\: (k,j) =(m,m+1). 
    \end{cases}
\end{equation}

Obviously, $f_{k,j} \in L_1[0,1]$, $k,j = \overline{1,n}$. Define the quasi-derivatives
\begin{equation} \label{quasi}
y^{[0]} := y, \quad y^{[k]} = (y^{[k-1]})' - \sum_{j = 1}^k f_{k,j} y^{[j-1]}, \quad k = \overline{1,n},
\end{equation}
and the domain
\begin{equation} \label{defDF}
\mathcal D_F = \{ y \colon y^{[k]} \in W_1^1[0,1], \, k = \overline{0, n-1} \} \subset W_2^m[0,1].
\end{equation}

The results of \cite{MS16} imply the following proposition on the regularization of the differential expression $\ell_n(y)$ of even order. 

\begin{prop} \label{prop:reg}
For any $y \in \mathcal D_F$, the function $\ell_n(y)$ is regular and $\ell_n(y) = y^{[n]}$.
\end{prop}

Note that the matrix function $F(x)$ constructed by the formulas \eqref{Seven} coincide with the associated matrix of \cite{MS16}. In order to obtain $F(x)$, we first represent $\ell_n(y)$ in terms of the bilinear form \eqref{quad} with the matrix $Q(x)$ and then use $Q(x)$ to find $F(x)$. The formulas \eqref{Seven} for constructing $F(x)$ by using $Q(x)$ have been obtained in \cite{Vlad17}. For our purposes, it is convenient to use the bilinear form with the matrix $Q(x)$ as an intermediate step.

\subsection{Class $\mathfrak F_n$} \label{sub:Fn}

In this subsection, we define the class $\mathfrak F_n$, which is the class of matrix functions $F(x)$ associated with differential expressions of form \eqref{defln}, as it will be shown in Subsection~\ref{sub:main}. Here, we study the properties of the quasi-derivatives and of the domain $\mathcal D_F$ generated by $F \in \mathfrak F_n$.

Define the spaces of matrix functions
\begin{align*}
\mathfrak Q_n := \bigl\{ Q(x) = [q_{r,j}]_{r,j = 0}^m \colon & q_{r,j} \in L_1[0,1], \, r,j = \overline{0, m-1}, \\ & q_{r,m}, q_{m,r} \in L_2[0,1], \, r = \overline{0,m-1}, \, q_{m,m} = 0 \bigr\}, \\
\mathfrak F_n := \bigl\{ F(x) = [f_{k,j}]_{k,j = 1}^n \colon & f_{k,j} \in L_1[0,1], \, f_{m,j}, \, f_{k,m+1} \in L_2[0,1], \, k = \overline{m+1, n}, \, j = \overline{1,m}, \\ & f_{k,j} = 0, \, k < m \:\: \text{or} \:\: j > m+1 \:\: \text{or} \:\: (k,j) =(m,m+1) \bigr\}.
\end{align*}

The structure of the spaces $\mathfrak Q_n$ and $\mathfrak F_n$ can be symbolically presented as follows:
\begin{align*}
n = 2 \colon & \quad 
Q = \begin{bmatrix}
L_1 & L_2 \\
L_2 & 0
\end{bmatrix}, \quad
F = \begin{bmatrix}
        L_2 & 0 \\
        L_1 & L_2
    \end{bmatrix}, \\
n = 4 \colon & \quad
Q = \begin{bmatrix}
        L_1 & L_1 & L_2 \\
        L_1 & L_1 & L_2 \\
        L_2 & L_2 & 0
    \end{bmatrix}, \quad
F = \begin{bmatrix}
        0 & 0 & 0 & 0 \\
        L_2 & L_2 & 0 & 0 \\
        L_1 & L_1 & L_2 & 0 \\
        L_1 & L_1 & L_2 & 0 
    \end{bmatrix}, \\
n = 6 \colon & \quad
Q = \begin{bmatrix}
        L_1 & L_1 & L_1 & L_2 \\
        L_1 & L_1 & L_1 & L_2 \\
        L_1 & L_1 & L_1 & L_2 \\
        L_2 & L_2 & L_2 & 0
    \end{bmatrix}, \quad
F = \begin{bmatrix}
        0 & 0 & 0 & 0 & 0 & 0 \\
        0 & 0 & 0 & 0 & 0 & 0 \\
        L_2 & L_2 & L_2 & 0 & 0 & 0\\
        L_1 & L_1 & L_1 & L_2 & 0 & 0\\
        L_1 & L_1 & L_1 & L_2 & 0 & 0\\
        L_1 & L_1 & L_1 & L_2 & 0 & 0        
    \end{bmatrix}.
\end{align*}

Clearly, in the Mirzoev-Shkalikov regularization described in Subsection~\ref{sub:MS}, $Q \in \mathfrak Q_n$ and $F \in \mathfrak F_n$. Furthermore, one can easily check that the mapping $\mathscr S_n \colon \mathfrak Q_n \to \mathfrak F_n$ defined by the formulas \eqref{Seven} is a bijection. The inverse mapping $Q = \mathscr S_n^{-1}(F)$ is given by the formulas
$$
\begin{cases}
    q_{j-1,m} := (-1)^{m+1} f_{m,j}, \quad j = \overline{1, m}, \\
    q_{m,2m-k} := (-1)^{k+1} f_{k,m+1}, \quad k = \overline{m+1, 2m}, \\
    q_{j-1,2m-k} := (-1)^{k+1} (f_{k,j} - f_{k,m+1} f_{m,j}), \quad k = \overline{m+1,2m}, \, j = \overline{1,m}, \\
    q_{m,m} := 0.
\end{cases}
$$

For each fixed $F \in \mathfrak F_n$, we can define the quasi-derivatives $y^{[k]}$ by \eqref{quasi} and the domain $\mathcal D_F$ by \eqref{defDF}. Let us study some of their properties.

\begin{prop} \label{prop:relFQ}
Let $Q(x)$ be a matrix function of $\mathfrak Q_n$ and $F := \mathscr S_n(Q)$. 
Then, for any $y \in \mathcal D_F$, the following relation holds:
\begin{equation} \label{relFQ}
(y^{[n]}, z) = (-1)^m ( y^{(m)}, z^{(m)} ) + \sum_{r,j = 0}^m (q_{r,j} y^{(r)}, z^{(j)}), \quad z \in \mathfrak D.
\end{equation}
\end{prop}

Proposition~\ref{prop:relFQ} follows from the general construction of Vladimirov~\cite{Vlad17}. The proof can be found in \cite{Bond23-mmas}. In particular, if the matrix function $Q(x)$ is constructed by \eqref{defQ}, then Propositions~\ref{prop:Q} and~\ref{prop:relFQ} together imply Proposition~\ref{prop:reg}.

Note that the domain $\mathcal D_F$ is a Banach space with the norm
$$
\| y \|_{\mathcal D_F} := \sum_{k = 0}^{n-1} \| y^{[k]} \|_{W_1^1[0,1]}.
$$
Consider the norms\footnote{Our norm in $W_p^s[0,1]$ differs from the standard one 
$\| y \|_{W_p^s[0,1]} = \left( \sum_{k = 0}^s \| y^{(k)} \|_{L_p[0,1]}^p \right)^{1/p}$. However, this does not influence the results, 
since these norms are equivalent.}
$$
\| y \|_{W_p^s[0,1]} = \sum_{k = 0}^s \| y^{(k)} \|_{L_p[0,1]}, \quad \| y \|_{L_p[0,1]} = \left( \int_0^1 |y(x)|^p \, dx \right)^{1/p}.
$$

\begin{lem} \label{lem:embed}
For every $F \in \mathfrak F_n$, the Banach space $\mathcal D_F$ is continuously and densely embedded in $W_2^m[0,1]$.
\end{lem}

\begin{proof}
For $s = \overline{m,n-1}$, consider the Banach spaces
$$
\mathcal D_F^s := \{ y \colon y^{[k]} \in W_1^1[0,1], \, k = \overline{0, s} \}
$$
with the corresponding norms
$$
\| y \|_{\mathcal D_F^s} := \sum_{k = 0}^s \| y^{[k]} \|_{W_1^1[0,1]}.
$$
Clearly, $\mathcal D_F = \mathcal D_F^{n-1}$. Let us prove that (i) $\mathcal D_F^m$ is continuously and densely embedded in $W_2^m[0,1]$, (ii) $\mathcal D_F^s$ is continuously and densely embedded in $\mathcal D_F^{s-1}[0,1]$ for $s = \overline{m+1, n-1}$. Obviously, the assertions (i) and (ii) together yield the claim of the lemma.

Let $y \in \mathcal D_F^m$. In view of \eqref{quasi} and the structure of $F \in \mathfrak F_n$, we have
\begin{align*}
& y^{(k)} = y^{[k]} \in W_1^1[0,1], \quad k = \overline{0,m-1}, \\
& y^{(m)} = y^{[m]} + \sum_{j = 1}^m f_{m,j} y^{(j-1)} \in L_2[0,1].
\end{align*}
Therefore, $y \in W_2^m[0,1]$. Moreover, we have
\begin{align*}
\| y^{[k]} \|_{L_2[0,1]} & \le \| y^{[k]} \|_{W_1^1[0,1]}, \quad k = \overline{0,m}, \\
\| y^{(m)} \|_{L_2[0,1]} & \le \| y^{[m]} \|_{L_2[0,1]} + \sum_{j = 1}^m \| f_{m,j} \|_{L_2[0,1]} \| y^{(j-1)} \|_{W_1^1[0,1]} \\
& \le \| y^{[m]} \|_{W_1^1[0,1]} + C \sum_{j = 0}^{m-1} \| y^{[j]} \|_{W_1^1[0,1]}.
\end{align*}
Hence
$$
\| y \|_{W_2^m[0,1]} = \sum_{k = 0}^m \| y^{(k)} \|_{L_2[0,1]} \le
C \sum_{k = 0}^m \| y^{[k]} \|_{W_1^1[0,1]} = \| y \|_{\mathcal D_F^m}.
$$
Thus, the embedding $\mathcal D_F^m \subset W_2^m[0,1]$ is continuous.

Next, let us construct a sequence $\{ y_r \}_{r \ge 1} \subset \mathcal D_F^m$ that approximates a function $y \in W_2^m[0,1]$. Note that, for $y \in W_2^m[0,1]$, the quasi-derivatives $y^{[k]}$ are correctly defined for $k = \overline{0,m}$, and $y^{[m]} \in L_2[0,1]$. Since $W_1^1[0,1]$ is dense in $L_2[0,1]$, there exists a sequence $\{ h_r \}_{r \ge 1}$ in $W_1^1[0,1]$ such that $\| h_r - y^{[m]} \|_{L_2[0,1]} \to 0$ as $r \to \infty$. The relations \eqref{quad} for $k = \overline{0,m}$ can be rewritten as the first-order system
$$
Y_m' = (F_m(x) + J_m) Y_m + y^{[m]}(x) e_m, \quad x \in (0, 1),
$$
where $Y_m$ is the column vector $[y^{[j]}]_{j = 0}^{m-1}$, $F_m(x)$ and $J_m$ are the $(m \times m)$ upper left submatrices of $F(x)$ and $J$, respectively, the matrix $J$ was defined in \eqref{defJ}, $e_m$ is the $m$-th column of the unit matrix. Let us consider the analogous system
\begin{equation} \label{sysYmr}
Y_{m,r}' = (F_m(x) + J_m) Y_{m,r} + h_r(x) e_m, \quad x \in (0, 1), 
\end{equation}
with respect to an unknown vector $Y_{m,r}(x)$, $r \ge 1$. The initial value problem for \eqref{sysYmr} with the initial condition $Y_{m,r}(0) = Y_m(0)$ has the unique solution $Y_{m,r}(x) = [y_{m,r,k}(x)]_{k = 0}^{m-1}$ such that $y_{m,r,k} = y_{m,r,0}^{[k]} \in W_1^1[0,1]$ for $k = \overline{0,m-1}$,
$h_r = y_{m,r,0}^{[m]} \in W_1^1[0,1]$, and $\| y_{m,r,k} - y^{[k]} \|_{L_2[0,1]} \to 0$ as $r \to \infty$. In other words, $y_r := y_{m,r,0} \in \mathcal D_F^m$ and $\| y_r - y \|_{W_2^m[0,1]} \to 0$ as $r \to \infty$. Hence, $\mathcal D_F^m$ is dense in $W_2^m[0,1]$.

Now, consider the embedding $\mathcal D_F^s \subset \mathcal D_F^{s-1}$, which is, obviously, continuous: $\| y \|_{\mathcal D_F^{s-1}} \le \| y \|_{\mathcal D_F^s}$. Let $y \in \mathcal D_F^{s-1}$. Then, it follows from \eqref{quad} that $y^{[s]} \in L_1[0,1]$. Due to the density of $W_1^1[0,1]$ in $L_1[0,1]$, there exists a sequence $\{ h_r \}_{r \ge 1} \subset W_1^1[0,1]$ such that $\| h_r - y^{[s]} \|_{L_1[0,1]} \to 0$ as $r \to \infty$. Similarly to \eqref{sysYmr}, we construct the system
$$
Y_{s,r}' = (F_s(x) + J_s) Y_{s,r} + h_r(x) e_s, \quad x \in (0, 1), 
$$
and prove that the first entry $y_r := y_{s,r,0}$ of its solution $Y_{s,r}$ belongs to $\mathcal D_F^s$ and $\| y_r - y \|_{\mathcal D_F^{s-1}} \to 0$ as $r \to \infty$. Thus, the embedding $\mathcal D_F^s \subset \mathcal D_F^{s-1}$ is dense, which completes the proof.
\end{proof}

\subsection{Main results and proofs} \label{sub:main}

In this subsection, we obtain the main results of this paper on the regularization of even order differential expressions. Namely, we construct the class of all the matrices $F(x) \in \mathfrak F_n$ associated to the differential expression $\ell_n(y)$ with fixed coefficients $\mathcal T = (\tau_{\nu})_{\nu = 0}^{n-1}$ and study the properties of this class. We begin with the rigorous definition.

\begin{df} \label{def:ass}
The matrix function $F(x) \in \mathfrak F_n$ is called \textit{associated} with the differential expression $\ell_n(y)$ if $\ell_n(y) = y^{[n]}$ for every $y \in \mathcal D_F$, where the quasi-derivatives $y^{[k]}$ and $\mathcal D_F$ are defined by \eqref{quasi} and \eqref{defDF}, respectively, by using the entries $f_{k,j}$ of $F(x)$.
\end{df}

For each $\mathcal T = (\tau_{\nu})_{\nu = 0}^{n-1} \in \mathfrak T_n$, denote by $\mathfrak F(\mathcal T)$ the set of all the associated matrices for the differential expression $\ell_n(y)$ with the coefficients $\mathcal T$.
It can be easily shown that
$\mathfrak F(\mathcal T) \cap \mathfrak F(\tilde{\mathcal T}) = \varnothing$ if $\mathcal T \ne \tilde{\mathcal T}$. Here, we mean that $\mathcal T = \tilde{\mathcal T}$ for $\mathcal T = (\tau_{\nu})_{\nu = 0}^{n-1}$ and $\tilde{\mathcal T} = (\tilde \tau_{\nu})_{\nu = 0}^{n-1}$ of class $\mathfrak T_n$ if $\tau_{\nu} = \tilde \tau_{\nu}$ in $W_2^{-i_{\nu}}[0,1]$ for $\nu = \overline{0,n-1}$.
Indeed, if the matrix $F \in \mathfrak F_n$ is associated with two coefficient vectors $\mathcal T$ and $\tilde{\mathcal T}$, then we have $\ell_n(y) = \tilde{\ell}_n(y) = y^{[n]}$ for all $y \in \mathcal D_F$. By virtue of Lemma~\ref{lem:embed}, $\mathcal D_F$ is dense in $W_2^m[0,1]$. Consequently, $\ell_n(y) = \tilde{\ell}_n(y)$ for all $y \in W_2^m[0,1]$, which implies $\mathcal T = \tilde{\mathcal T}$.

By virtue of Proposition~\ref{prop:reg}, for every $\mathcal T \in \mathfrak T_n$, at least one associated matrix $F(x)$ exists, so $\mathfrak F(\mathcal T) \ne \varnothing$. This matrix is constructed as $F = \mathscr S_n(Q)$, where $Q = \mathscr Q_n(\Sigma(\mathcal T))$ (see \eqref{tausig} and \eqref{defQ}).
However, there exist other associated matrices.

\begin{thm} \label{thm:invF}
Every matrix $F(x)$ of class $\mathfrak F_n$ is associated with some differential expression $\ell_n(y)$, whose coefficients $\mathcal T = (\tau_{\nu})_{\nu = 0}^{n-1}$ belong to $\mathfrak T_n$. In other words,
$$
\mathfrak F_n = \bigcup_{\mathcal T \in \mathfrak T_n} \mathfrak F(\mathcal T).
$$
\end{thm}

In order to prove Theorem~\ref{thm:invF}, we need some auxiliary lemmas. Consider an arbitrary matrix function $F(x)$ of $\mathfrak F_n$.
Find $Q := \mathscr S_n^{-1}(F) \in \mathfrak Q_n$. Consider the constant matrices $\chi_{\nu,i}$ of size $(m + 1) \times (m + 1)$ defined by \eqref{defchi} for $\nu = \overline{0, n-1}$, $i = \overline{0,i_{\nu}}$. For example, for $n = 2$ and $n = 4$, we have
\begin{align} \label{chi2}
n = 2 \colon \quad &
\chi_{0,0} = 
\begin{bmatrix}
1 & 0 \\
0 & 0
\end{bmatrix}, \quad
\chi_{0,1} = 
\begin{bmatrix}
0 & 1 \\
1 & 0
\end{bmatrix}, \quad
\chi_{1,0} = 
\begin{bmatrix}
0 & 1 \\
-1 & 0
\end{bmatrix}, \\ \nonumber
n = 4 \colon \quad &  \chi_{0,0} = 
\begin{bmatrix}
1 & 0 & 0 \\
0 & 0 & 0 \\
0 & 0 & 0
\end{bmatrix}, \quad
\chi_{0,1} = 
\begin{bmatrix}
0 & 1 & 0 \\
1 & 0 & 0 \\
0 & 0 & 0
\end{bmatrix}, \quad
\chi_{0,2} = 
\begin{bmatrix}
0 & 0 & 1 \\
0 & 2 & 0 \\
1 & 0 & 0
\end{bmatrix}, \quad
\chi_{1,0} = 
\begin{bmatrix}
0 & 1 & 0 \\
-1 & 0 & 0 \\
0 & 0 & 0
\end{bmatrix}, \\ \nonumber &
\chi_{1,1} = 
\begin{bmatrix}
0 & 0 & 1 \\
0 & 0 & 0 \\
-1 & 0 & 0
\end{bmatrix}, \quad
\chi_{2,0} = 
\begin{bmatrix}
0 & 0 & 0 \\
0 & 1 & 0 \\
0 & 0 & 0
\end{bmatrix}, \quad
\chi_{2,1} = 
\begin{bmatrix}
0 & 0 & 0 \\
0 & 0 & 1 \\
0 & 1 & 0
\end{bmatrix}, \quad
\chi_{3,0} = 
\begin{bmatrix}
0 & 0 & 0 \\
0 & 0 & 1 \\
0 & -1 & 0
\end{bmatrix}.
\end{align}

\begin{lem} \label{lem:basis}
The matrices $\chi_{\nu,i}$, $\nu = \overline{0,n-1}$, $i = \overline{0,i_{\nu}}$, form a basis in the linear space
$$
\mathfrak M_n := \bigl\{ [a_{r,j}]_{r,j = 0}^m \colon a_{r,j} \in \mathbb R, \, a_{m,m} = 0 \bigr\}.
$$
\end{lem}

\begin{proof}
Due to \eqref{defchi}, every matrix $\chi_{\nu,i}$ has non-zero entries $\chi_{\nu,i; r,j}$ only on the diagonal $r + j = d$, where $d = \nu + i$. In general, for any diagonal with number $d = 0, 1, 2, \dots, n-1$, we have the number of the corresponding matrices $\chi_{\nu, i}$ equal to the length of this diagonal. Moreover, these matrices are linearly independent. This concludes the proof.
\end{proof}

It follows from Lemma~\ref{lem:basis} that any matrix function $Q \in \mathfrak Q_n$ admits the unique representation
\begin{equation} \label{Qtau}
Q(x) = \sum_{\nu = 0}^{n-1} \sum_{i = 0}^{i_{\nu}} \tau_{\nu,i}(x) \chi_{\nu,i},
\end{equation}
where
\begin{equation} \label{reqtau}
\tau_{\nu, i_{\nu}} \in L_2[0,1], \quad \tau_{\nu, i} \in L_1[0,1], \quad \nu = \overline{0, n-1}, \: i = \overline{0,i_{\nu}-1}.
\end{equation}

Note that the right-hand side of \eqref{defQ} is the special case of \eqref{Qtau} with $\tau_{\nu, i_{\nu}} = \sigma_{\nu}$ and $\tau_{\nu, i} = 0$ for $i < i_{\nu}$. Direct calculations prove the following lemma.

\begin{lem} \label{lem:findtau}
Let $Q(x)$ be given by formula \eqref{Qtau}, where $\tau_{\nu, i}$ are arbitrary functions satisfying \eqref{reqtau}. Then, for any $y \in W_2^m[0,1]$, the relation \eqref{quad} holds, where the coefficients $\mathcal T = (\tau_{\nu})_{\nu = 0}^{n-1} \in \mathfrak T_n$ of $\ell_n(y)$ are defined as follows:
\begin{equation} \label{findtau}
\tau_{\nu} := \sum_{i = 0}^{i_{\nu}} (-1)^i \tau_{\nu,i}^{(i)}, \quad \nu = \overline{0,n-1}.
\end{equation}
\end{lem}

\begin{proof}[Proof of Theorem~\ref{thm:invF}]
Let $F \in \mathfrak F_n$ and $Q = \mathscr S_n^{-1}(F)$. Then, $Q \in \mathfrak Q_n$, and so $Q(x)$ admits the unique representation \eqref{Qtau}. Using the coefficients $\tau_{\nu,i}$ of this representation, find $\tau_{\nu}$ by \eqref{findtau} and consider the differential expression $\ell_n(y)$ with the coefficients $\mathcal T = (\tau_{\nu})_{\nu = 0}^{n-1} \in \mathfrak T_n$. Using Lemma~\ref{lem:findtau} and Proposition~\ref{prop:relFQ}, we conclude that $\ell_n(y) = y^{[n]}$ in $\mathfrak D'$ for any $y \in \mathcal D_F$. Hence $F \in \mathfrak F(\mathcal T)$.
\end{proof}

\begin{cor} \label{cor:construct}
The set $\mathfrak F(\mathcal T)$ of associated matrices can be described constructively. Let $\mathcal T = (\tau_{\nu})_{\nu = 0}^{n-1} \in \mathfrak T_n$ be fixed. Choose arbitrary functions $\tau_{\nu,i} \in L_1[0,1]$ and constants $c_{\nu,i} \in \mathbb C$ for $\nu = \overline{0,n-1}$, $i = \overline{0,i_{\nu}-1}$. Taking formula \eqref{findtau} into account, find
\begin{equation} \label{findtaunu}
\tau_{\nu,i_{\nu}} := (-1)^{i_{\nu}}\left( \tau_{\nu} - \sum_{i = 0}^{i_{\nu}-1} (-1)^i \tau_{\nu,i}^{(i)} \right)^{(-i_{\nu})} + \sum_{i = 0}^{i_{\nu}-1} c_{\nu,i} x^i, \quad \nu = \overline{0,n-1}.
\end{equation}

The notation $y^{(-i)}$ in \eqref{findtaunu} is used for a fixed antiderivative of order $i$. For example, the choice of the antiderivative can be fixed by the conditions
$$
\int_0^1 x^k w(x) \, dx = 0, \quad k = \overline{0,i-1}, \quad w(x) = y^{(-i)}(x).
$$

Clearly, $\tau_{\nu,i_{\nu}} \in L_2[0,1]$. Then, using the functions $\tau_{\nu,i}$, $\nu = \overline{0,n-1}$, $i = \overline{0,i_{\nu}}$, find $Q(x)$ by \eqref{Qtau} and $F = \mathscr S_n(Q)$. Denote the matrix $F(x)$ constructed by this algorithm as $\mathscr F(\mathcal T, (\tau_{\nu,i},c_{\nu,i}))$.  The functions $\tau_{\nu,i}$ for $i < i_{\nu}$ and the constants $c_{\nu,i}$ can be chosen arbitrarily as parameters.
Thus
\begin{equation} \label{descrFT}
\mathfrak F(\mathcal T) = \bigl\{ \mathscr F(\mathcal T, (\tau_{\nu,i},c_{\nu,i}))\colon \tau_{\nu,i} \in L_1[0,1], \, c_{\nu,i} \in \mathbb C, \, \nu = \overline{0,n-2}, \, i = \overline{0,i_{\nu}-1} \bigr\}.
\end{equation}
\end{cor}

\begin{thm} \label{thm:dom}
Let $\mathcal T \in \mathfrak T_n$ be fixed. Then $\mathcal D_F = \mathcal D_{\tilde F}$ for any $F, \tilde F \in \mathfrak F(\mathcal T)$.
\end{thm}

\begin{proof}
Let $\mathcal T$, $F(x) = [f_{k,j}]_{k,j = 1}^n$, and $\tilde F(x) = [\tilde f_{k,j}]_{k,j = 1}^n$ satisfy the hypothesis of the theorem.
In this proof, we use the notations $y^{[k]}_F$ and $y^{[k]}_{\tilde F}$ for the quasi-derivatives defined by formulas \eqref{quasi} by the entries of the matrices $F(x)$ and $\tilde F(x)$, respectively.
Denote $\hat f_{k,j} := f_{k,j} - \tilde f_{k,j}$ and
\begin{align} \nonumber
\hat y^{[j]} & := 0, \quad j = \overline{0,m-1}, \\ \label{defyhat}
\hat y^{[k]} & := (\hat y^{[k-1]})' - \sum_{j = 1}^m \hat f_{k,j} y^{(j-1)} - \hat f_{k,m+1} y^{[m]}_F - \tilde f_{k,m+1} \hat y^{[m]}, \quad k = \overline{m,n}.
\end{align}

In order to prove the theorem, it is sufficient to show that, for any $y \in W_2^m[0,1]$,
\begin{equation} \label{haty}
\hat y^{[j]} \in W_1^1[0,1], \quad j = \overline{m,n-1}.
\end{equation}

Indeed, if $y \in \mathcal D_F$, then 
$y_{\tilde F}^{[k]} = y_F^{[k]} + \hat y^{[k]}$. Consequently, \eqref{haty} implies that $y_{\tilde F}^{[k]} \in W_1^1[0,1]$, $k = \overline{0,n-1}$, so $y \in \mathcal D_{\tilde F}$. 

Put $Q := \mathscr S_n^{-1}(F)$, $\tilde Q := \mathscr S_n^{-1}(\tilde F)$, and $\hat q_{r,j} := q_{r,j} - \tilde q_{r,j}$. Let $y \in W_2^m[0,1]$
Using \eqref{Seven} and \eqref{defyhat}, we derive
\begin{align*}
\hat y^{[m]} & = -\sum_{j = 1}^m \hat f_{m,j} y^{(j-1)} = \sum_{j = 1}^{m-1} \hat q_{j,m} y^{(j)}, \\
\hat y^{[k]} & = (\hat y^{[k - 1]})' - \sum_{j = 1}^m (\hat f_{k,j} - \hat f_{k,m+1} f_{m,j} - \tilde f_{k,m+1} \hat f_{m,j}) y^{(j-1)} - \hat f_{k,m+1} y^{(m)}  \\
& = (\hat y^{[k-1]})' + (-1)^k \sum_{j = 0}^m \hat q_{j, 2m-k} y^{(j)}, \quad k = \overline{m+1,n}.
\end{align*}
Hence, for $z \in \mathfrak D$ and $k = \overline{0,m-1}$, we have
$$
(\hat y^{[m+k]}, z) = -(\hat y^{[m+k-1]}, z') + (-1)^{m+k} \sum_{j = 0}^m (\hat q_{j,m-k} y^{(j)}, z).
$$
By induction, we obtain
$$
(\hat y^{[m+k]}, z) = (-1)^{m+k} \sum_{l = 0}^k \sum_{j = 0}^m (\hat q_{j,m-k+l} y^{(j)}, z^{(l)}).
$$
The change of summation indices implies
\begin{equation} \label{sm1}
(\hat y^{[m+k]}, z) = (-1)^{m+k} \sum_{r = 0}^m \sum_{j = m-k}^m (\hat q_{r,j} y^{(r)}, z^{(j - (m-k))}).
\end{equation}

By virtue of Proposition~\ref{prop:relFQ} and Corollary~\ref{cor:construct}, for the both matrices $Q(x)$ and $\tilde Q(x)$, the relation \eqref{quad} holds with the same differential expression $\ell_n(y)$. Hence
\begin{gather} \nonumber
\sum_{r,j = 0}^m (\hat q_{r,j} y^{(r)}, g^{(j)}) = 0, \quad y \in W_2^m[0,1], \quad g \in \mathfrak D, \\ \label{hatquad}
\sum_{r = 0}^m \sum_{j = m-k}^m (\hat q_{r,j} y^{(r)}, g^{(j)}) = -\sum_{r = 0}^m \sum_{j = 0}^{m-k-1} (\hat q_{r,j} y^{(r)}, g^{(j)}).
\end{gather}

It can be shown that the relation \eqref{hatquad} is valid not only for $g \in \mathfrak D$ but also for $g = z^{(-s)}$, $z \in \mathfrak D$, $s \in \mathbb N$, where
$$
z^{(0)} := z, \quad z^{(-s)} := \int_0^x z^{(-(s-1))}(t) \, dt.
$$
Antiderivatives of $\mathfrak D$-functions are infinitely differentiable and have a support $[a,b] \subset (0,1]$. In other words, the derivatives $g^{(k)}(1)$, $k = 0, 1, \dots$ can be non-zero. In order to overcome this difficulty, one can extend the interval $(0,1)$ to $(0, 1+\eps)$, $\eps > 0$, put $\hat \tau_{\nu} = 0$, $\hat q_{r,j} = 0$ on $(1,1+\eps)$, extend $y$ and $g$ so that $y \in W_2^m[0,1+\eps]$ and $g \in C_0^{\infty}(0, 1+\eps)$, respectively. Consequently, we can apply the relation \eqref{hatquad} to the function $g := z^{(-(m-k))}$ in \eqref{sm1}:
$$
(\hat y^{[m+k]}, z) = (-1)^{m+k+1} \sum_{r = 0}^m \sum_{j = 0}^{m-k-1} (\hat q_{r,j} y^{(r)}, z^{(j - (m-k))}).
$$
Integration by parts implies
$$
(\hat y^{[m+k]}, z) = \sum_{r = 0}^m \sum_{j = 0}^{m-k-1}(-1)^{j+1} \bigl( (\hat q_{r,j} y^{(r)})^{\langle j - (m-k)\rangle}, z \bigr),
$$
where 
$$
y^{\langle 0 \rangle} := y, \quad y^{\langle -s \rangle}(x) = -\int_x^1 y^{\langle -(s-1) \rangle}(t) \, dt, \quad s \ge 1.
$$
Hence
\begin{equation} \label{sm2}
\hat y^{[m+k]} = \sum_{r = 0}^m \sum_{j = 0}^{m-k-1} (-1)^{j+1} (\hat q_{r,j} y^{(r)})^{\langle j - (m-k)\rangle}, \quad k = \overline{0,m-1}.
\end{equation}

Recall that $Q, \tilde Q \in \mathfrak Q_n$ and $y \in W_2^m[0,1]$. Therefore, $\hat q_{r,j} y^{(r)} \in L_1[0,1]$ for all $r$ and $j$ in \eqref{sm2}. Furthermore, $j - (m-k) \le -1$. This implies \eqref{haty} and so concludes the proof.
\end{proof}

\section{Odd order} \label{sec:odd}

In this section, we provide the regularization results, analogous to the ones in Section~\ref{sec:even}, for odd orders.

Consider the differential expression \eqref{defln} of order $n = 2m+1$, $m \in \mathbb N$, with $\mathcal T := (\tau_{\nu})_{\nu = 0}^{n-1} \in \mathfrak T_n$, where
$$
\mathfrak T_n := \bigl\{ \mathcal T = (\tau_{\nu})_{\nu = 0}^{n-1} \colon \tau_{\nu} \in W_1^{-i_{\nu}}[0,1], \, \nu = \overline{0,n-1} \bigr\},
$$
and the singularity orders $(i_{\nu})_{\nu = 0}^{n-1}$ are defined by \eqref{defi}. In other words, the relations \eqref{tausig} are valid for some $\sigma_{\nu} \in L_1[0,1]$.
If $y \in W_1^m[0,1]$, then $\ell_n(y) \in \mathfrak D'$ and the following relation holds:
\begin{equation} \label{quad-odd}
(\ell_n(y), z) = (-1)^m ( y^{(m+1)}, z^{(m)} ) + \sum_{r,j = 0}^m (q_{r,j} y^{(r)}, z^{(j)}), \quad z \in \mathfrak D,
\end{equation}
where $Q(x) = [q_{r,j}]_{r,j = 0}^m$ and the matrices $\chi_{\nu,i}$ are defined by \eqref{defQ} and \eqref{defchi}, respectively.

Construct the matrix function $F(x) = [f_{k,j}]_{k,j =1}^n$ by the rule $F = \mathscr S_n(Q)$ given by the formulas
\begin{equation} \label{Sodd}
\begin{cases}
f_{k,j} := (-1)^k q_{j-1, 2m+1-k}, \quad k = \overline{m+1, 2m+1}, \, j = \overline{1, m+1}, \\
f_{k,j} := 0, \quad \text{otherwise}.
\end{cases}
\end{equation}

For this matrix $F(x)$, Proposition~\ref{prop:reg} holds (see \cite{MS19} and Theorem 2.2 in \cite{Bond23-mmas}).

Define the spaces of matrix functions
\begin{align*}
\mathfrak Q_n := \bigl\{Q(x) = [q_{r,j}]_{r,j = 0}^m \colon & q_{r,j} \in L_1[0,1], \, r,j = \overline{0,m} \bigr\}, \\
\mathfrak F_n := \bigl\{ F(x) = [f_{k,j}]_{k,j = 1}^n \colon & f_{k,j} \in L_1[0,1], \, k = \overline{m+1,2m+1}, \, j = \overline{1,m+1}, \\ & f_{k,j} = 0, \, k < m+1 \: \text{or} \: j > m+1 \bigr\}.
\end{align*}

The structure of the spaces $\mathfrak F_n$ can be symbolically presented as follows:
\begin{equation*}
n = 3 \colon \quad
F = \begin{bmatrix}
        0 & 0 & 0 \\
        L_1 & L_1 & 0 \\
        L_1 & L_1 & 0 
    \end{bmatrix}, \qquad
n = 5 \colon \quad
F = \begin{bmatrix}
        0 & 0 & 0 & 0 & 0 \\
        0 & 0 & 0 & 0 & 0 \\
        L_1 & L_1 & L_1 & 0 & 0 \\
        L_1 & L_1 & L_1 & 0 & 0 \\
        L_1 & L_1 & L_1 & 0 & 0 
    \end{bmatrix}.
\end{equation*}

As well as in the even-order case, the mapping $\mathscr S_n \colon \mathfrak Q_n \to \mathfrak F_n$ defined by \eqref{Sodd} is a bijection. The results of Mirzoev and Shkalikov \cite{MS19} imply that the matrix function $F = \mathscr S_n(\mathscr Q_n(\Sigma))$, where $\Sigma = \Sigma(\mathcal T)$, is associated with the odd-order differential expression $\ell_n(y)$ in the sense of Definition~\ref{def:ass}.

Similarly to the even-order case, denote by $\mathfrak F(\mathcal T)$ the set of associated matrices for $\ell_n(y)$ with the coefficients $\mathcal T$. Theorem~\ref{thm:invF} is also valid for odd $n$. Indeed, one can easily show that the matrices $\chi_{\nu,i}$, $\nu = \overline{0,n-1}$, $i = \overline{0,i_{\nu}}$ form a basis in the linear space
$$
\mathfrak M_n := \bigl\{ [a_{r,j}]_{r,j = 0}^m \colon a_{r,j} \in \mathbb R \bigr\}.
$$

Note that the only difference in the matrices $\chi_{\nu,i}$ between the cases $n = 2m$ and $n = 2m+1$ is the additional matrix with the unit entry $a_{m,m}$ for the odd order. For example, for $n = 3$, we have the following matrices (compare with \eqref{chi2}):
$$
\chi_{0,0} = 
\begin{bmatrix}
1 & 0 \\
0 & 0
\end{bmatrix}, \quad
\chi_{0,1} = 
\begin{bmatrix}
0 & 1 \\
1 & 0
\end{bmatrix}, \quad
\chi_{1,0} = 
\begin{bmatrix}
0 & 1 \\
-1 & 0
\end{bmatrix},
\quad 
\chi_{2,0} =
\begin{bmatrix}
0 & 0 \\
0 & 1
\end{bmatrix}.
$$

Any matrix $Q \in \mathfrak Q_n$ admits the unique representation \eqref{Qtau}, where $\tau_{\nu,i} \in L_1[0,1]$, $\nu = \overline{0,n-1}$, $i = \overline{0,i_{\nu}}$. On the other hand, if $Q(x)$ is given by \eqref{Qtau}, then, for any $y \in W_1^m[0,1]$, the relation \eqref{quad-odd} holds, where the coefficients $\mathcal T = (\tau_{\nu})_{\nu = 0}^{n-1}$ are defined by \eqref{findtau} and $\mathcal T \in \mathfrak T_n$. This implies the assertion of Theorem~\ref{thm:invF} for odd orders. 

The set $\mathfrak F(T)$ is described by formula \eqref{descrFT}. The functions $\tau_{\nu,i_{\nu}}$ constructed by \eqref{findtaunu} belong to $L_1[0,1]$, $\nu = \overline{0,n-1}$. Theorem~\ref{thm:dom} also holds for odd orders.

\begin{remark}
In the inverse problem theory (see \cite{Bond21, Bond23-mmas, Bond22, Bond23-results}), differential expressions of form \eqref{defln} with $\tau_{n-1} = 0$ are considered. Denote 
$$
\mathfrak T_n^0 := \bigl\{ \mathcal T = (\tau_{\nu})_{\nu = 0}^{n-1} \in \mathfrak T_n \colon \tau_{n-1} = 0 \bigr\}.
$$
Since the set $\mathfrak F(\mathcal T)$ of associated matrices for $\mathcal T \in \mathfrak T_n^0$ is constructed according to Corollary~\ref{cor:construct}, we obtain
$$
\bigcup_{\mathcal T \in \mathfrak T_n^0} \mathfrak F(\mathcal T) = \mathfrak F_n^0,
$$
where
$$
\mathfrak F_n^0 := \bigl\{ F \in \mathfrak F_n \colon \mbox{trace}\,(F) = 0 \bigr\}
$$
for both even and odd values of $n$.
\end{remark}

\section{Inverse problems} \label{sec:ip}

In this section, inverse spectral problems are investigated for the differential equation generated by the expression $\ell_n(y)$. We define the spectral characteristics, study their dependence on the associated matrix, and prove the uniqueness theorems for the inverse problems (Theorems~\ref{thm:uniq} and \ref{thm:sd}). In addition, we compare our novel results with the known uniqueness theorems from the previous studies \cite{Bond21, Bond22, Bond23-mmas}.

Consider the differential expression $\ell_n(y)$ with coefficients $\mathcal T \in \mathfrak T_n^0$ (i.e. $\tau_{n-1} = 0$). Let $F(x)$ be a fixed associated matrix for $\ell_n(y)$, that is, $F \in \mathfrak F(\mathcal T) \subset \mathfrak F_n^0$. Define the quasi-derivatives $y^{[k]}$ and the domain $\mathcal D_F$ by \eqref{quasi} and \eqref{defDF}, respectively. According to Definition~\ref{def:ass}, for any $y \in \mathcal D_F$, we have $\ell_n(y) \in L_1[0,1]$. Below, we call a function $y$ a \textit{solution} of the equation
\begin{equation} \label{eqv}
\ell_n(y) = \la y, \quad x \in (0, 1),
\end{equation}
if $y \in \mathcal D_F$ and the relation \eqref{eqv} holds a.e. on $(0,1)$.

Denote by $\{ C_k(x, \la) \}_{k = 1}^n$ and by $\{ \Phi_k(x, \la) \}_{k = 1}^n$ the solutions of equation \eqref{eqv} satisfying the initial conditions
\begin{equation} \label{icC}
C_k^{[j-1]}(0,\la) = \de_{k,j}, \quad j = \overline{1, n},
\end{equation}
and the boundary conditions
\begin{equation} \label{bcPhi}
\Phi_k^{[j-1]}(0,\la) = \de_{k,j}, \quad j = \overline{1, k}, \quad
\Phi_k^{[n-s]}(1,\la) = 0, \quad s = \overline{k+1,n},
\end{equation}
respectively, where $\de_{k,j}$ is the Kronecker delta.
It has been shown in \cite{Bond21} that the initial value problem solutions $C_k(x, \la)$ exist and are unique. The boundary value problem solutions $\Phi_k(x, \la)$ are uniquely defined for all complex $\la$ except for a countable set. Moreover, for each fixed $x \in [0,1]$ and $j = \overline{1,n}$, the quasi-derivatives $C_k^{[j-1]}(x, \la)$ are entire in $\la$ and $\Phi_k^{[j-1]}(x, \la)$ are meromorphic in $\la$. Furthermore, the matrix functions $C(x, \la) = [C_k^{[j-1]}(x, \la)]_{j,k = 1}^n$ and $\Phi(x, \la) = [\Phi_k^{[j-1]}(x, \la)]_{j,k = 1}^n$ are related as follows:
\begin{equation} \label{defM}
\Phi(x, \la) = C(x, \la) M(\la),
\end{equation}
where the matrix function $M(\la) = [M_{j,k}(\la)]_{j,k = 1}^n$ is called \textit{the Weyl-Yurko matrix} of equation~\eqref{eqv}. 

It can be shown similarly to \cite{Yur02, Bond21, Bond23-mmas} that $M(\la)$ is a unit lower-triangular matrix. Furthermore, its non-trivial entries $M_{j,k}(\la)$ for $j > k$ are meromorphic functions with countable sets of poles. More precisely, the poles of $M_{j,k}(\la)$ coincide with eigenvalues of the boundary value problem $\mathcal L_k$ for equation \eqref{eqv} with the boundary conditions
$$
y^{[j-1]}(0) = 0, \quad j = \overline{1,k}, \qquad
y^{[n-s]}(1) = 0, \quad s = \overline{k+1,n},
$$
which correspond to \eqref{bcPhi}.

Now, suppose that we have two associated matrices $F(x)$ and $\tilde F(x)$ of $\mathfrak F(\mathcal T)$. Denote the quasi-derivatives constructed by $F(x)$ and $\tilde F(x)$ by $y^{[k]}_F$ and $y^{[k]}_{\tilde F}$, respectively.
We agree that, if a certain object $\al$ is related to $F(x)$, then the symbol $\tilde \al$ with tilde will denote the analogous object related to $\tilde F(x)$. The following theorem establishes the relation between the Weyl-Yurko matrices corresponding to different associated matrices.

\begin{thm} \label{thm:L}
Suppose that $\mathcal T \in \mathfrak T_n^0$, $F, \tilde F \in \mathfrak F(\mathcal T)$. Then $M(\la) = L \tilde M(\la)$, where $L \in \mathfrak L_n$,
$$
\mathfrak L_n = \bigl\{ L = [l_{j,k}]_{j,k = 1}^n \in \mathbb C^{n \times n} \colon l_{j,k} = \delta_{j,k} \:\: \text{for} \:\: j \le n-m \:\: \text{or} \:\: k > m\bigr\}.
$$
\end{thm}

Thus, the matrices of $\mathfrak L_n$ have the following structure:
$$
n = 3 \colon \quad
\begin{bmatrix}
1 & 0 & 0 \\
0 & 1 & 0 \\
* & 0 & 1 
\end{bmatrix}, \quad 
n = 4 \colon \quad
\begin{bmatrix}
1 & 0 & 0 & 0\\
0 & 1 & 0 & 0\\
* & * & 1 & 0 \\
* & * & 0 & 1
\end{bmatrix}.
$$

\begin{proof}[Proof of Theorem~\ref{thm:L}]
By virtue of Theorem~\ref{thm:dom}, $\mathcal D_F = \mathcal D_{\tilde F}$. Therefore, it can be easily seen that
\begin{align} \label{sm3}
\bigl\{ y \in \mathcal D_F \colon y^{[j-1]}_F(0) = \de_{k,j}, \: j = \overline{1,k} \bigr\} & =
\bigl\{ y \in \mathcal D_{\tilde F} \colon y^{[j-1]}_{\tilde F}(0) = \de_{k,j}, \: j = \overline{1,k} \bigr\}, \\ \label{sm4}
\bigl\{ y \in \mathcal D_F \colon y^{[j-1]}_F(1) = 0, \: j = \overline{1,k} \bigr\} & =
\bigl\{ y \in \mathcal D_{\tilde F} \colon y^{[j-1]}_{\tilde F}(1) = 0, \: j = \overline{1,k} \bigr\}
\end{align}
for each $k = \overline{1,n}$. Hence $\Phi(x, \la) \equiv \tilde \Phi(x, \la)$. For $C(x, \la)$ and $\tilde C(x, \la)$, the relations \eqref{sm3} and \eqref{sm4} imply that
$$
\tilde C_k(x, \la) = C_j(x, \la) + \sum_{j = k+1}^n l_{j,k} C_j(x, \la), \quad l_{j,k} \in \mathbb C,
$$
that is, $\tilde C(x, \la) \equiv C(x, \la) L$, where $L = [l_{j,k}]_{j,k = 1}^n$ is a unit lower-triangular matrix. Using the special structure of the matrices $F(x)$ and $\tilde F(x)$ of class $\mathfrak F_n$, namely, the relations $f_{k,j} = 0$ for $k < n - m - 1$ or $j > m+1$, we prove that $L \in \mathfrak L_n$. Using the relations $\tilde \Phi(x, \la) \equiv \Phi(x, \la)$, $\tilde C(x, \la) \equiv C(x, \la) L$, and \eqref{defM}, we arrive at the assertion of the theorem.
\end{proof}

The inverse result is also valid:

\begin{thm} \label{thm:uniq}
Suppose that $\mathcal T = (\tau_{\nu})_{\nu = 0}^{n-1}$ and $\tilde{\mathcal T} = (\tilde \tau_{\nu})_{\nu = 0}^{n-1}$ belong to $\mathfrak T_n^0$,
$M(\la)$ and $\tilde M(\la)$ are the Weyl-Yurko matrices defined by the associated matrices $F \in \mathfrak F(\mathcal T)$ and $\tilde F \in \mathfrak F(\tilde{\mathcal T})$, respectively, and $M(\la) = L \tilde M(\la)$, where $L \in \mathfrak L_n$. Then $\mathcal T = \tilde{\mathcal T}$. Thus, the Weyl-Yurko matrix $M(\la)$ known up to a multiplier $L \in \mathfrak L_n$ uniquely specifies the coefficients $\mathcal T \in \mathfrak T_n^0$ of the differential expression $\ell_n(y)$.
\end{thm}

\begin{proof}
Introduce the matrix of spectral mappings 
$$
P(x, \la) = \Phi(x, \la) (\tilde \Phi(x, \la))^{-1}.
$$
Using \eqref{defM} and the relation $M(\la) = L \tilde M(\la)$, we obtain
$$
P(x, \la) = C(x, \la) M(\la) (\tilde M(\la))^{-1} (\tilde C(x, \la))^{-1} = C(x, \la) L (\tilde C(x, \la))^{-1}.
$$
Hence $P(x, \la)$ is entire in $\la$ for each fixed $x \in [0,1]$. Then, similarly to the proof of Theorem~2 in \cite{Bond21}, we show that, for each fixed $x \in [0,1)$, $P(x, \la)$ is a constant unit triangular matrix $P(x) = [p_{k,j}(x)]_{k,j = 1}^n$, which satisfies the relation
\begin{equation} \label{relP}
P'(x) + P(x) \tilde F(x) = F(x) P(x), \quad x \in (0,1).
\end{equation}

Let us prove that \eqref{relP} implies $\mathcal T = \tilde{\mathcal T}$. For definiteness, suppose that $n = 2m$. The odd order case can be investigated analogously. By considering the first $(m-1)$ rows and the last $(m-1)$ columns of \eqref{relP}, we deduce that $p_{k,j} = 0$ $(j < k)$ for $k = \overline{1,m}$ and $j = \overline{m+1, 2m}$, respectively. From the relations for $k = \overline{m, 2m}$ and $j = \overline{1,m+1}$, we derive 
\begin{equation} \label{sm5}
p'_{k,j} + p_{k,j-1} + (\tilde f_{k,j} - \tilde f_{k,m+1} \tilde f_{m,j}) = p_{k+1,j} + (f_{k,j} - f_{k,m+1} f_{m,j}), \quad k = \overline{m+1, 2m}, \: j = \overline{1, m}.
\end{equation}

Using \eqref{Seven}, we transform \eqref{sm5} into the system
\begin{gather} \label{eqr}
    r'_{l,s} + r_{l-1,s} + r_{l,s-1} = \hat q_{l,s}, \quad l,s = \overline{0, m},  \\ \label{bcr}
    r_{-1,s} = r_{s,-1} = r_{m,s} = r_{s,m} = 0, \quad s = \overline{0, m},
\end{gather}
where $\hat q_{l,s} = q_{l,s} - \tilde q_{l,s}$, 
$[q_{l,s}]_{l,s = 0}^m := \mathscr S_n^{-1}(F)$, $[\tilde q_{l,s}]_{l,s = 0}^m := \mathscr S_n^{-1}(\tilde F)$, and $r_{j-1, 2m-k} := (-1)^{k+1} p_{k,j}$, $k = \overline{m,2m}$, $j = \overline{1, m+1}$.

For $y \in W_2^m[0,1]$ and $z \in \mathfrak D$, using \eqref{eqr}, we derive
\begin{align} \nonumber
\sum_{l,s = 0}^m (\hat q_{l,s} y^{(l)}, z^{(s)}) = & \sum_{l,s} (r'_{l,s} y^{(l)}, z^{(s)}) + \sum_{l,s} (r_{l-1,s} y^{(l)}, z^{(s)}) + \sum_{l,s} (r_{l,s-1} y^{(l)}, z^{(s)}) \\ \nonumber
= & \sum_{l,s}((r_{l,s}y^{(l)})', z^{(s)}) - \sum_{l,s}(r_{l,s}y^{(l+1)}, z^{(s)}) \\ \label{sm6} & + \sum_{l,s} (r_{l,s} y^{(l+1)}, z^{(s)}) + \sum_{l,s} (r_{l,s} y^{(l)}, z^{(s+1)}) = 0.
\end{align}

Here, we have applied the index shift, the integration by parts and have taken the boundary conditions \eqref{bcr} into account. Using Lemma~\ref{lem:findtau} and Corollary~\ref{cor:construct}, we obtain the relations
\begin{align*}
(\ell_n(y), z) & = (-1)^m ( y^{(m)}, z^{(m)} ) + \sum_{l,s = 0}^m (q_{l,s} y^{(l)}, z^{(s)}), \\
(\tilde \ell_n(y), z) & = (-1)^m ( y^{(m)}, z^{(m)} ) + \sum_{l,s = 0}^m (\tilde q_{l,s} y^{(l)}, z^{(s)}).
\end{align*}
Combining them with \eqref{sm6}, conclude that $(\ell_n(y), z) = (\tilde \ell_n(y), z)$ for all $y \in W_2^m[0,1]$ and $z \in \mathfrak D$. This implies $\mathcal T = \tilde{\mathcal T}$.
\end{proof}

Let us compare Theorem~\ref{thm:uniq} with the following uniqueness result of \cite{Bond21}. Introduce the space
$$
\mathfrak S_n^0 = \left\{ \Sigma = (\sigma_{\nu})_{\nu = 0}^{n-1} \colon 
\begin{array}{ll}
\sigma_{\nu} \in L_2[0,1] & \text{if $n$ is even}, \\ 
\sigma_{\nu} \in L_1[0,1] & \text{if $n$ is odd},
\end{array}
\: \nu = \overline{0,n-2}, \, \sigma_{n-1} = 0\right\}.
$$

\begin{prop}[\cite{Bond21}] \label{prop:uniq}
Suppose that $\Sigma = (\sigma_{\nu})_{\nu = 0}^{n-1}$ and $\tilde \Sigma = (\tilde \sigma_{\nu})_{\nu = 0}^{n-1}$ belong to $\mathfrak S_n^0$, $F = \mathscr S_n(\mathscr Q_n(\Sigma))$, $\tilde F = \mathscr S_n(\mathscr Q_n(\tilde \Sigma))$, $M(\la)$ and $\tilde M(\la)$ are the Weyl-Yurko matrices of $F$ and $\tilde F$, respectively, and $M(\la) = \tilde M(\la)$. Then $\Sigma = \tilde \Sigma$, that is, $\sigma_{\nu}(x) = \tilde \sigma_{\nu}(x)$ a.e. on $(0,1)$.
\end{prop}

Note that, in Proposition~\ref{prop:uniq}, the fixed Mirzoev-Shkalikov construction of the associated matrices is assumed. Then, the antiderivatives $(\sigma_{\nu})_{\nu  = 0}^{n-2}$ are uniquely determined by the Weyl-Yurko matrix. Theorem~\ref{thm:uniq} corresponds to a different inverse problem. If the Weyl-Yurko matrix is known up to a factor $L \in \mathfrak L_n$, then it uniquely specifies the distribution coefficients $(\tau_{\nu})_{\nu = 0}^{n-2}$ independently of the choice of the associated matrix. These are two different results. It is worth mentioning that, in \cite{Bond21}, Proposition~\ref{prop:uniq} has been proved for a more general type of separated boundary conditions than \eqref{bcPhi}. 
The conditions \eqref{bcPhi} have the lowest possible orders, so they are the most simple ones. In other cases, boundary conditions contain constant coefficients, which either can be recovered or have to be given a priori (see \cite{Bond23-mmas}). For the general separated boundary conditions, Theorem~\ref{thm:uniq} does not hold (see the example in Subsection~\ref{sub:ex2}). 

Next, consider the inverse problem by the discrete spectral data, which was studied in \cite{Bond22}. Denote by $\Lambda$ the poles of the Weyl-Yurko matrix $M(\la)$. We will write $M \in W$ if all the poles of $M(\la)$ are simple. Then, the Laurent series has the form
$$
M(\la) = \frac{M_{\langle -1 \rangle}(\la_0)}{\la - \la_0} + M_{\langle 0 \rangle}(\la_0) + M_{\langle 1 \rangle}(\la_0)(\la - \la_0) + \dots, \quad \la_0 \in \Lambda.
$$
Define \textit{the weight matrices} as follows:
\begin{equation} \label{defN}
\mathcal N(\la_0) := \left( M_{\langle 0 \rangle}(\la_0) \right)^{-1} M_{\langle -1\rangle}(\la_0), \quad \la_0 \in \Lambda.
\end{equation}

In view of Theorem~\ref{thm:L}, the weight matrices are uniquely specified by the coefficients $\mathcal T$ and do not depend on the associated matrix $F \in \mathfrak F(\mathcal T)$. The inverse is also true:

\begin{thm} \label{thm:sd}
Suppose that $\mathcal T = (\tau_{\nu})_{\nu = 0}^{n-1}$ and $\tilde{\mathcal T} = (\tilde \tau_{\nu})_{\nu = 0}^{n-1}$ belong to $\mathfrak T_n^0$,
$M(\la)$ and $\tilde M(\la)$ are the Weyl-Yurko matrices defined by the associated matrices $F \in \mathfrak F(\mathcal T)$ and $\tilde F \in \mathfrak F(\tilde{\mathcal T})$, respectively, $M, \tilde M \in W$, and the corresponding spectral data sets $\{ \la_0, \mathcal N(\la_0) \}_{\la_0 \in \Lambda}$ and $\{ \la_0, \tilde{\mathcal N}(\la_0) \}_{\la_0 \in \tilde \Lambda}$ are equal to each other. Then $\mathcal T = \tilde{\mathcal T}$. Thus, the spectral data $\{ \la_0, \mathcal N(\la_0) \}_{\la_0 \in \Lambda}$ uniquely specify the coefficients $\mathcal T$ of the differential expression $\ell_n(y)$.
\end{thm}

For $n = 3$, Theorem~\ref{thm:sd} has been proved in \cite{Bond22}. For the general case, Theorem~\ref{thm:sd} is proved analogously to Theorem~\ref{thm:uniq}, since the matrix of spectral mappings $P(x, \la)$ is entire in $\la$ due to Lemma~9 in \cite{Bond22}.

\section{Example} \label{sec:ex}

In this section, as an example, we consider the case $n = 2$. 

\subsection{Regularization} \label{sub:ex1}

In this case, $m = 1$ and the differential expression \eqref{defln} takes the form 
\begin{equation} \label{defl2}
\ell_2(y) = y'' - (\tau_1(x) y)' - \tau_1(x) y' + \tau_0(x) y.
\end{equation}

Thus, $\mathcal T = (\tau_0, \tau_1)$, $i_0 = 1$, $i_1 = 0$,
the space $\mathfrak T_2$ is defined by the conditions $\tau_0 \in W_2^{-1}[0,1]$, $\tau_1 \in L_2[0,1]$, and $\Sigma = (\sigma_0, \sigma_1)$, $\tau_0 = -\sigma_0'$, $\sigma_1 = \tau_1$, so $\sigma_j \in L_2[0,1]$, $j = 0,1$. The antiderivative $\sigma_0$ of $-\tau_0$ can be chosen uniquely up to an additive constant.

Suppose that $y \in W_2^1[0,1]$ and $z \in \mathfrak D$. Calculations show that
\begin{align*}
& (y'', z) = -(y',z'), \\
& (- (\tau_1 y)' - \tau_1 y', z) = (\sigma_1 y, z') - (\sigma_1 y', z), \\
& (\tau_0 y, z) = (-(\sigma_0 y)' + \sigma_0 y', z) = (\sigma_0 y, z') + (\sigma_0 y', z).
\end{align*}
By summation, we arrive at the relation \eqref{quad}:
$$
(\ell_2(y), z)  = -(y', z') + (q_{0,0} y, z) + (q_{1,0} y', z) + (q_{0,1} y, z'), \quad y \in W_2^1[0,1], \quad z \in \mathfrak D, 
$$
where
$$
Q(x) = \begin{bmatrix}
        q_{0,0} & q_{0,1} \\
        q_{1,0} & 0
    \end{bmatrix} = \sigma_0 \chi_{0,1} + \sigma_1 \chi_{1,0}
    = \begin{bmatrix}
        0 & \sigma_0 + \sigma_1 \\
        \sigma_0 - \sigma_1 & 0
    \end{bmatrix}.
$$
(The matrices $\chi_{0,1}$ and $\chi_{1,0}$ are given by \eqref{chi2}).

Using formulas \eqref{Seven}, we obtain
$$
F(x) = \mathscr S_2(Q) = \begin{bmatrix}
                            q_{0,1} & 0 \\
                            -(q_{0,0} + q_{0,1} q_{1,0}) & -q_{1,0}
                        \end{bmatrix}
        = \begin{bmatrix}
            \sigma_1 + \sigma_0 & 0 \\
            \sigma_1^2 - \sigma_0^2 & \sigma_1 - \sigma_0
        \end{bmatrix}.
$$

The matrix-function $F(x)$ is the special case of the regularization matrix by Mirzoev and Shkalikov \cite{MS16}. It coincides with the associated matrices from \cite{Mir14} and \cite{VS15}. For $\mathcal T \in \mathfrak T_2^0$ (i.e. $\tau_1=0$), we obtain the differential expression $y'' + \tau_0 y$, $\tau_0 \in W_2^{-1}[0,1]$ and the associated matrix
\begin{equation} \label{regsi}
F(x) = \begin{bmatrix}
            \sigma_0 & 0 \\
            -\sigma_0^2 & -\sigma_0
        \end{bmatrix},
\end{equation}
which was widely used for investigation of direct and inverse Sturm-Liouville problems (see, e.g., \cite{SS03, HM-sd, HM-2sp}).

Now, proceed to the construction of the set $\mathfrak F(\mathcal T)$ of all the associated matrices for the differential expression \eqref{defl2} with $\mathcal T \in \mathfrak T_2$. Any matrix function $Q \in \mathfrak Q_2$ admits the representation \eqref{Qtau}:
\begin{equation} \label{findQ2}
Q(x) = \tau_{0,0} \begin{bmatrix}
                        1 & 0 \\
                        0 & 0
                    \end{bmatrix}
        + \tau_{0,1} \begin{bmatrix}
                        0 & 1 \\
                        1 & 0
                    \end{bmatrix} 
        + \tau_{1,0} \begin{bmatrix}
                        0 & 1 \\
                        -1 & 0
                    \end{bmatrix},
\end{equation}
where $\tau_{0,0} \in L_1[0,1]$ and $\tau_{0,1}, \tau_{1,0} \in L_2[0,1]$. The matrix $Q(x)$ is related to the differential expression $\ell_2(y)$ with the coefficients $\tau_0 = \tau_{0,0} - \tau_{0,1}'$, $\tau_1 = \tau_{1,0}$ (see \eqref{findtau}).

Suppose that $\tau_0 \in W_2^{-1}[0,1]$ and $\tau_1 \in L_2[0,1]$ are given.
Choose an arbitrary function $\tau_{0,0} \in L_1[0,1]$ and a constant $c_{0,0} \in \mathbb C$. Find
$$
\tau_{0,1} = \left( \tau_{0,0} - \tau_0 \right)^{(-1)} + c_{0,0}, \quad
\tau_{1,0} = \tau_1,
$$
construct $Q(x)$ by \eqref{findQ2} and $F(x) = \mathscr S_2(Q)$, which is the associated matrix for $\ell_2(y)$. By choosing different $\tau_{0,0} \in L_1[0,1]$ and $c_{0,0} \in \mathbb C$, we will obtain different associated matrices for the same differential expression. In particular, for $y'' + \tau_0 y$, $\tau_0 \in W_2^{-1}[0,1]$, all the associated matrices can be represented as
\begin{equation} \label{reg2}
F(x) = \begin{bmatrix}
            \sigma & 0 \\
            -\sigma^2 & -\sigma
        \end{bmatrix} + 
        \begin{bmatrix}
            0 & 0 \\
            r & 0
        \end{bmatrix},
\quad \sigma \in L_2[0,1], \: r \in L_1[0,1], \quad \tau_0 = -(\sigma' + r). 
\end{equation}
Clearly, $r \in L_1[0,1]$ can be chosen arbitrarily. After that, the function $\sigma = (\tau_0 - r)^{(-1)}$ is determined uniquely up to an additive constant. Thus, every $F \in \mathfrak T_2^0$ can be represented as the sum of the two associated matrices that are usually used for the Sturm-Liouville expressions $y'' - \tau(x) y$, $\tau = \sigma' \in W_2^1[0,1]$, and $y'' - r(x) y$, $r \in L_1[0,1]$.

\subsection{Inverse problems} \label{sub:ex2}

Proceed to inverse spectral problems for the Sturm-Liouville equation
\begin{equation} \label{StL}
y'' - q(x) y = \la y, \quad x \in (0,1), \quad q = \sigma' \in W_2^{-1}[0,1].
\end{equation}
Let us formulate some known uniqueness results and compare them to the results of Section~\ref{sec:ip}. 

The associated matrix \eqref{regsi} ($\sigma_0 = \sigma$) produces the quasi-derivative $y^{[1]} = y' - \sigma y$. Following the classical inverse problem theory (see, e.g., \cite{FY01}), introduce the main spectral characteristics. Let $\{ \la_n \}_{n \ge 1}$ and $\{ \mu_n \}_{n \ge 1}$ be the eigenvalues of the boundary value problems for equation \eqref{StL} with the boundary conditions $y(0) = y(1) = 0$ and $y^{[1]}(0) = y(1) = 0$, respectively. Denote by $S(x, \la)$ and $C(x, \la)$ the solutions of equation \eqref{StL} satisfying the initial conditions
$$
S(0, \la) = C^{[1]}(0,\la) = 0, \quad S^{[1]}(0,\la) = C(0,\la) = 1.
$$

Obviously, the eigenvalues $\{ \la_n \}_{n \ge 1}$ and $\{ \mu_n \}_{n \ge 1}$ coincide with the zeros of the characteristic functions $S(1, \la)$ and $C(1,\la)$, respectively. In the case of simple eigenvalues $\{ \la_n \}_{n \ge 1}$, introduce \textit{the weight numbers} $\al_n := \int_0^1 y_n^2(x) \, dx$, $n \ge 1$, where $y_n(x) = S(x, \la_n)$ are the corresponding eigenfunctions. (In the case of multiple eigenvalues, one can use the generalized weight numbers as in \cite{But07, BSY13}). Furthermore, define \textit{the Weyl function} $m(\la) := -\dfrac{C(1,\la)}{S(1,\la)}$, which is meromorphic in the $\la$-plane.

In the case of regular potential $q \in L_1[0,1]$, each of the following three types of the spectral data uniquely specifies $q$:

\medskip

(i) the two spectra $\{ \la_n, \mu_n \}_{n \ge 1}$; 

\smallskip

(ii) the eigenvalues $\{ \la_n \}_{n \ge 1}$ and the weight numbers $\{ \al_n \}_{n \ge 1}$ (if the eigenvalues are simple);

\smallskip

(iii) the Weyl function $m(\la)$.

\medskip

Moreover, the spectral data (i)--(iii) uniquely determine each other. In the case of distribution potential $q \in W_2^{-1}[0,1]$, the situation is slightly different:

\begin{enumerate}
\item The two spectra $\{ \la_n, \mu_n \}_{n \ge 1}$ and the Weyl function $m(\la)$ uniquely specify each other. Indeed, on the one hand, $\{ \la_n \}_{n \ge 1}$ and $\{ \mu_n \}_{n \ge 1}$ coincide with the poles and the zeros of $m(\la)$, respectively. On the other hand, the characteristic functions $S(1,\la)$ and $C(1,\la)$ can be constructed as infinite produces by their zeros, and so $m(\la)$ can be found.
\item The weight numbers $\{ \al_n \}_{n \ge 1}$ are uniquely specified by $m(\la)$: $\al_n^{-1} = \Res_{\la = \la_n} m(\la)$, while $m(\la)$ is uniquely determined by $\{ \la_n, \al_n \}_{n \ge 1}$ up to an additive constant (see \cite{Bond21-mn}).
\end{enumerate}

From the inverse problem viewpoint, we have the following uniqueness results (see \cite{HM-sd, HM-2sp}):

\begin{enumerate}
\item $\{ \la_n, \mu_n \}_{n \ge 1}$ or $m(\la)$ uniquely specify $\sigma(x)$.
\item $\{ \la_n, \al_n \}_{n \ge 1}$ uniquely specify $q(x)$ or $\sigma(x) + c$, where $c$ is an arbitrary constant.
\end{enumerate}

The first result corresponds to Proposition~\ref{prop:uniq} and the second one, to Theorem~\ref{thm:sd}. Indeed, due to Section~\ref{sec:ip}, the Weyl-Yurko matrix for equation \eqref{StL} has the form
$$
M(\la) = \begin{bmatrix}
            1 & 0 \\
            m(\la) & 1
        \end{bmatrix},
$$
where $m(\la)$ is the Weyl function. Note that the definition of $m(\la)$ is strongly connected with the regularization matrix \eqref{regsi} ($\sigma_0 = \sigma$), while the antiderivative $\sigma(x)$ of $q(x)$ is defined up to a constant $c$. Thus, $m(\la)$ depends on $c$. Consequently, one can uniquely recover $\sigma(x)$ from $m(\la)$ as in Proposition~\ref{prop:uniq}. Anyway, we can consider other associated matrices generated by \eqref{reg2}. By virtue of Theorem~\ref{thm:L}, the Weyl-Yurko matrices $M(\la)$ and $\tilde M(\la)$, which are obtained from different associated matrices $F(x)$ and $\tilde F(x)$ in the second-order case, are related as follows:
$$
\begin{bmatrix}
            1 & 0 \\
            m(\la) & 1
        \end{bmatrix}
= 
\begin{bmatrix}
1 & 0 \\
l_{2,1} & 1
\end{bmatrix}
\begin{bmatrix}
            1 & 0 \\
            \tilde m(\la) & 1
        \end{bmatrix} \quad \Leftrightarrow \quad
m(\la) = l_{2,1} + \tilde m(\la).
$$

Thus, the assumption that $M(\la)$ is given up to a factor $L \in \mathfrak L_2$ actually means that the Weyl function $m(\la)$ is given up to an additive constant $c$. Then, by using the given Weyl function, we cannot uniquely determine $\sigma(x)$ but, by virtue of Theorem~\ref{thm:uniq}, can uniquely determine $q(x)$. This corresponds to the inverse problem by the spectral data $\{ \la_0, \mathcal N(\la_0) \}_{\la_0 \in \Lambda}$, which in the second-order case has the form
$$
\Lambda = \{ \la_n \}_{n \ge 1}, \quad \mathcal N(\la_n) = \begin{bmatrix}
                                                                0 & 0 \\
                                                                \al_n^{-1} & 0 
                                                          \end{bmatrix}, \: n \ge 1.
$$
Hence, the spectral data $\{ \la_n, \al_n \}_{n \ge 1}$ do not depend on the choice of the associated matrix. Theorem~\ref{thm:sd} imply that $\{ \la_n, \al_n \}_{n \ge 1}$ uniquely specify $q(x)$ but, obviously, it cannot uniquely specify $\sigma(x)$.

The situation changes for different types of boundary conditions. In particular, the spectral data of the Sturm-Liouville equation \eqref{StL} with the Robin-type boundary conditions
$$
y^{[1]}(0) - h y(0) = 0, \quad y^{[1]}(1) + H y(1) = 0, \quad h, H \in \mathbb C,
$$
are invariant with respect to the shift $\sigma := \sigma + c$, $h := h - c$, $H := H + c$, $c \in \mathbb C$. Therefore, it is natural to fix $h = 0$. Then, the corresponding three types of spectral data (two spectra, eigenvalues and weight numbers, and the Weyl function) uniquely determine each other as well as the coefficients $\sigma(x)$ and $H$ (see \cite{HM-sd, HM-2sp}). However, the other types of inverse problems, which consist in determining $q(x)$ (but not $\sigma(x)$) and correspond to Theorems~\ref{thm:uniq} and~\ref{thm:sd}, cannot be considered for the Robin-type boundary conditions, because the associated matrix (roughly speaking, $\sigma(x)$) is related to the coefficients $h$ and $H$.

Thus, the described examples for $n = 2$ show that the both types of inverse problems, that is, the recovery of $\mathcal T$ and of $\Sigma$, generalize the classical problem statements. But, for distribution coefficients, these two types are different and the both are worth being studied.

\section{Conclusion} \label{sec:concl}

In this paper, for each $n \ge 2$, we have considered the class $\mathfrak F_n$, which contains the matrices of Mirzoev and Shkalikov \cite{MS16, MS19} associated with the differential expressions $\ell_n(y)$. We have shown that, every matrix function $F \in \mathfrak F_n$ is associated with some differential expression $\ell_n(y)$ with coefficients $\mathcal T = (\tau_{\nu})_{\nu = 0}^{n-1} \in \mathfrak T_n$. Furthermore, we have constructively described the family $\mathfrak F(\mathcal T) \subset \mathfrak F_n$ of all the associated matrices for fixed $\mathcal T$. In addition, we have proved that $\mathcal D_F = \mathcal D_{\tilde F}$ for $F, \tilde F \in \mathfrak F(\mathcal T)$. The both even and odd order cases have been studied. Moreover, we applied this construction to the inverse problem theory. The uniqueness theorems have been proved for inverse spectral problems of a new type.

Our results have the following \textbf{advantages} over the previous studies:

\begin{enumerate}
\item We have investigated various matrices associated with the same differential expression, while the previous studies provide only specific constructions of associated matrices.
\item We have studied a novel class of inverse spectral problems which consist in the recovery of distributional coefficients $\mathcal T$ independently of the associated matrix. In the previous works for higher-order differential operators with distribution coefficients, spectral data were connected with a fixed associated matrix. However, the both types of inverse problems generalize the classical problem statements and are worth being investigated.
\end{enumerate}

In the future, our results can be applied to studying various spectral theory issues for differential operators with distribution coefficients, because, for investigation of different spectral properties, it can be convenient to use different associated matrices. In particular, the results of this paper can be used for obtaining solvability and stability conditions for inverse problems.

\medskip

{\bf Funding.} This work was supported by Grant 21-71-10001 of the Russian Science Foundation, https://rscf.ru/en/project/21-71-10001/ % (accesed on ?? July 2023).

\medskip

%{\bf Data Availability Statement:} Not applicable.

%\medskip

% {\bf Conflicts of Interest:} The author declares that this paper has no conflict of interest.

%\medskip

\noindent Natalia Pavlovna Bondarenko \\

\noindent 1. Department of Mechanics and Mathematics, Saratov State University, \\
Astrakhanskaya 83, Saratov 410012, Russia, \\

\noindent 2. Department of Applied Mathematics and Physics, Samara National Research University, \\
Moskovskoye Shosse 34, Samara 443086, Russia, \\

\noindent 3. Peoples' Friendship University of Russia (RUDN University), \\
6 Miklukho-Maklaya Street, Moscow, 117198, Russia, \\

\noindent e-mail: {\it bondarenkonp@info.sgu.ru}

\end{document}